\documentclass[12pt,a4paper, final, twoside]{article}

\usepackage{amsmath}
\usepackage{fancyhdr}
\usepackage{amsthm}
\usepackage{amsfonts}
\usepackage{amssymb}
\usepackage{amscd}
\usepackage{latexsym}
\usepackage{amsthm}
\usepackage{graphicx}
\usepackage{graphics}
\usepackage{afterpage}
\usepackage[colorlinks=true, urlcolor=blue,  linkcolor=black, citecolor=black]{hyperref}
\usepackage{color}
\theoremstyle{plain} 
\newtheorem{thm}{Theorem}[section]

\newtheorem{lem}[thm]{Lemma}

\theoremstyle{proposition}
\newtheorem{pr}{Proposition}[section]

\newtheorem{cor}[thm]{Corollary}

\theoremstyle{definition}
\newtheorem{defi}{Definition}[section]
\theoremstyle{remark}

\numberwithin{equation}{section}
\begin{document}
\hyphenpenalty=100000

\begin{center}
{\Large {\textbf{\\  R-boundedness Approach to linear third differential equations in a UMD Space}}}\\[5mm]
{\large \textbf{{Bahloul Rachid}$^\mathrm{{\bf \color{red}{1}}}$\footnote{\emph{{1} : E-mail address : bahloul33r@hotmail.com}} } \\[1mm]
{\footnotesize $^\mathrm{}$ 
      }\\[3mm]}

\end{center}

\begin{flushleft}\footnotesize \it \textbf{$^\mathrm{{\bf \color{red}{1}}}$ Department of Mathematics, Faculty of Sciences and Technology, {\bf Fez}, Morocco.
}\\[3mm]
\end{flushleft}

\begin{center}\textbf {ABSTRACT}\end{center} 
{\footnotesize {{The aim of this work is to study the existence of a periodic solutions of  third order differential equations $z'''(t) = Az(t) +  f(t)$  with the periodic condition $x(0) = x(2\pi), x'(0) = x'(2\pi)$ and $x''(0) = x''(2\pi)$. Our approach is based on the  R-boundedness and $L^{p}$-multiplier
of linear operators.}}}

\afterpage{
\fancyhead{} \fancyfoot{}
\fancyfoot[R]{\footnotesize\thepage}
\fancyhead[R]{\scriptsize \it{ }
 }}

\tableofcontents

\section{Introduction}
Motivated by the fact that  functional differential equations  arise in many areas of applied mathematics, this type of
equations has received much attention in recent years. In particular, the problem of existence of periodic solutions, has been considered by several authors. We refer the readers to papers [\cite{1}, \cite{6}, \cite{9}, \cite{14}] and the references listed therein for informations on this subject.\\
In this work, we study the existence of periodic solutions for the following differential equations 

\begin{eqnarray}\label{e2}
\left\{
\begin{array}{ccccc}
\displaystyle{z'''(t) = Az(t) +  f(t)}\\\\
x(0) = x(2\pi), x'(0) = x'(2\pi) \ \text{and} \ x''(0) = x''(2\pi).
\end{array}
\right.
\end{eqnarray}
where  $A : D(A) \subseteq  X \rightarrow X$ is a linear closed operator on  Banach space ($X, \left\|.\right\|$) and $\alpha$ can be any real number and $f \in L^{p}(\mathbb{T}, X)$ for all $p \geq 1$. \\
Hale \cite{18} and Webb \cite{23} firstly studied the first order delay equation:
\begin{equation}\label{eee}
u'(t) = Au(t) + F(u_{t}) + f(t),
\end{equation}
B\'atkai et al. \cite{5} obtained results on the hyperbolicity of delay equations using the theory of operatorvalued
Fourier multipliers. Bu \cite{8} has studied $C^{\alpha}$-maximal regularity for the problem (\ref{eee}) on R. Recently, Lizama
\cite{14} obtained necessary and sufficient conditions for the first order delay equation (\ref{eee}) to have $L^{p}$-maximal
regularity using multiplier theorems on $L^{p}$-($\mathbb{T}$;X), and $C^{\alpha}$-maximal regularity of the corresponding equation
on the real line has been studied by Lizama and Poblete \cite{15}.\\ 
Arendt \cite{1} gave necessary and sufficient conditions for the existence of periodic solutions of  the following evolution equation.
$$\displaystyle{\frac{d}{dt}x(t)= Ax(t)+f(t)}\;\; \text{for}\;\; t \in \mathbb{R},$$
where $A$ is a closed linear operator on an UMD-space $Y$. \\
Hernan et al \cite{9}, studied the existence of periodic solutions for the class of linear abstract neutral functional differential equation described in the following form:
\[\frac{d}{dt}[x(t) - Bx(t - r)]= Ax(t)+G(x_{t})+f(t)\;\;\text{ for}\;\; t \in \mathbb{R}\]
where $A: D(A) \rightarrow X$ and $B : D(B) \rightarrow X$ are closed linear operator such that $D(A) \subset D(B)$ and $G \in B(L^{p}([- 2\pi,\ \ 0],\ \ X);\ \ X)$.\\
Bahaj et al \cite{3} studied the existence of periodic solution of second degenerate differential equation described in the following form:
\[(Mx)''(t) + Ax(t) + G(x_{t}) = f(t)\;\;\text{ for}\;\; t \in \mathbb{R}\]
where $A: D(A) \rightarrow X$ and $M : D(M) \rightarrow X$ are closed linear operator such that $D(A) \subset D(M)$ and $G \in B(L^{p}([- 2\pi,\ \ 0],\ \ X);\ \ X)$.\\
The organization of this work is as follows:
In section $2$, collects definitions and basic properties of R-bounded, UMD space and Fourier multipliers,
In section $3$, we study the sufficient Conditions For the Periodic solutions of Eq. {\bf \color{red}{\eqref{e2}}},
In section $4$, we establish the periodic solution for the equation (\ref{e2}) of this work solely in terms of a property of R-boundedness for the sequence of operators $-ik^{3 }(-ik^{3}+(\alpha - 1)A)^{-1}$. We optain that the following assertion are equivalent in UMD space :
\begin{align*}
&1) D_{A}=\frac{d^{3}}{dt^{3}} - A : H^{3,p}(\mathbb{T}, X) \cap L^{p}(D(A), X) \rightarrow L^{p}(\mathbb{T}, X) \ \text{is an isomorphism}.\\
&2) \sigma_{\mathbb{Z}}(\Delta)= \phi, \ \left\{ -ik^{3}\Delta^{-1}_{k}    \right\}_{k \in \mathbb{Z}} \text{is R-bounded}. \\
&3) \forall f \in L^{p}(\mathbb{T}, X) \text{\ there exists a unique $2\pi$-periodic strong $L^{p}$-solution  of \ }  Eq. {\bf \eqref{e2}}.
\end{align*}
In section $5$, we propose an application.
In section $6$, we give the conclusion.

\section{Vector-valued space and preliminaries}
Let $X$ be a Banach Space. Firstly, we denote By $\mathbb{T}$ the group defined as the quotient $\mathbb{R}/2 \pi \mathbb{Z}$. There is an identification between functions on $\mathbb{T}$ and $2\pi$-periodic functions on $\mathbb{R}$. We consider the interval $[0, 2\pi$) as a model for $\mathbb{T}$ .\\
Given $1 \leq p < \infty $, we denote by $L^ {p}(\mathbb{T}; X)$ the space of $2\pi$-periodic locally $p$-integrable functions from $\mathbb{R}$ into $X$, with the norm:
\[\left\|f\right\|_{p}: =\left( \int_{0}^{2\pi} \left\|f(t)\right\|^{p}dt \right)^{1/p}\]
For  $f \in L^{p}(\mathbb{T}; X)$, we denote by $\hat{f}(k)$, $k \in \mathbb{Z}$ the $k$-th Fourier coefficient of $f$ that is defined by:

\[\hat{f}(k) = \frac{1}{2\pi}\int_{0}^{2\pi}e^{-ikt}f(t)dt\;\; \text{for}\;\; k \in \mathbb{Z} \ \ \text{and} \  \ t \in \mathbb{R}.\]
For $1 \leq p < \infty$, the periodic vector-valued space is defined by

\begin{equation}
H^{1,p}(\mathbb{T}; X) = \left\{ u \in L^{p}(\mathbb{T}, X)   : \exists  v \in  L^{p}(\mathbb{T}, X) , \hat{v}(k) = ik \hat{u}(k) \  \text{for all} \  \  k \in \mathbb{Z}       \right\}
\end{equation}

\subsection{UMD Space}
\begin{defi}
Let $\varepsilon  \in  ]0, 1[$ and $1 < p < \infty$. Define the operator $H_{\varepsilon}$   by:

for all $f \in L^{p}(\mathbb{R}; X)$  
\[(H_{\varepsilon}f)(t):=\frac{1}{\pi} \int_{\varepsilon < \left|s\right|  < \frac{1}{\epsilon}}^{}\frac{f(t - s)}{s}ds\]

if $\lim\limits_{\varepsilon \rightarrow 0} H_{\varepsilon}f:= Hf$ exists in $L^{p}(\mathbb{R}; X)$ Then $Hf$ is called the Hilbert transform of $f$ on $L^{p}(\mathbb{R}, X)$.
\end{defi}

\begin{defi} {\bf \color{green}{\cite{1}}} \\
A Banach space $X$ is said to be UMD space if the Hilbert transform is bounded on $L^{p}(\mathbb{R};\ \ X)$ for all $1 < p < \infty$.
\end{defi}

\subsection{$R$-bounded and $L^{p}$-multiplier}
Let $X$ and $Y$ be  Banach spaces. Then $B(X,Y)$ denotes, the space of bounded linear operators from X to Y.
\begin{defi} {\bf \color{green}{\cite{1}}}\\
A family of operators $T=(T_{j})_{j \in \mathbb{N}^{\ast}}\subset B(X,Y)$ is called $R$-bounded (\textbf{ Rademacher bounded or randomized bounded}), if there is a constant $C > 0$ and
$p \in [1, \infty)$ such that for each $n \in N, T_{j} \in $T$, x_{j}\in X$ and for all independent, symmetric, $\left\{-1,1\right\}$-valued random variables $r_{j}$ on a probability space ($\Omega, M, \mu$) the inequality
$$\left\|\sum_{j=1}^{n} r_{j} T_{j} x_{j}\right\|_{L^{p}(0,1;Y)}\leq C \left\|\sum_{j=1}^{n} r_{j} x_{j}\right\|_{L^{p}(0,1;X)}$$
is valid. The smallest $C$ is called $R$-bounded of  $(T_{j})_{j \in \mathbb{N}^{\ast}}$ and it is denoted by $R_{p}$($T$).
\end{defi}

\begin{defi}{\bf \color{green}{\cite{1}}}\\
For $1\leq p < \infty$ , a sequence $\left\{M_{k}\right\}_{k \in \mathbb{Z}} \subset \mathbf{B}(X,Y)$ is said to be an $L^{p}$-multiplier if for each $f \in L^{p}(\mathbb{T}, X)$, there exists $ u\in$ $L^{p}(\mathbb{T}, Y)$ such that $\hat{u}(k) = M_{k}\hat{f}(k)$ for all $k \in \mathbb{Z}$.
\end{defi}

\begin{pr} \label{p}$[{\bf \color{green}{\cite{1}}},\ \ Proposition\ \ 1.11]$\\
Let $X$ be a Banach space and  $\left\{M_{k}\right\}_{k \in \mathbb{Z}}$ be an $L^{p}$-multiplier, where $1 \leq p < \infty$. Then the set $\left\{M_{k}\right\}_{k \in \mathbb{Z}}$ is $R$-bounded.
\end{pr}

\begin{thm}\label{t1} \textbf{(Marcinkiewicz operator-valud multiplier Theorem)}.\\
Let $X$, $Y$ be UMD spaces and  $\left\{M_{k}\right\}_{k \in \mathbb{Z}} \subset B(X, Y)$. If the sets $\left\{M_{k}\right\}_{k \in \mathbb{Z}}$ and
$\left\{k(M_{k+1}-M_{k})\right\}_{k \in \mathbb{Z}}$ are \\$R$-bounded, then $\left\{M_{k}\right\}_{k \in \mathbb{Z}}$ is an $L^{p}$-multiplier for 
$ 1 < p < \infty $.
\end{thm}

\begin{thm}\label{t2}$( Fejer's \  theorem)$\\
Let $f \in L^{p}(\mathbb{T},X)$. Then $$f = \lim_{n \rightarrow \infty}\sigma_{n}(f)$$ in $L^{p}(\mathbb{T},X)$ where 
$$\sigma_{n}(f):=\frac{1}{n+1}\sum_{m=0}^{n}\sum_{k=-m}^{m}e_{k}\hat{f}(k)$$ with $e_{k}(t):=e^{ikt}$.
\end{thm}

\begin{lem}\label{lem2} {\bf \color{green}{\cite{1}}}.
Let $f,g \in L^{p}(\mathbb{T}; X)$. If $\hat{f}(k) \in D(A)$ and $A\hat{f}(k) = \hat{g}(k)$ for all $k \in \mathbb{Z}$ Then $$f(t) \in D(A) \;\;\text{and}\;\; Af(t) =g(t)\;\; \text{for all }\;\;t \in [0, 2\pi].$$
\end{lem}

\section{Sufficient Conditions For the Periodic solutions of Eq. {\bf \color{red}{\eqref{e2}}}}
In this section, we will give conditions which guarantee the periodic solution of the some second differential equation.
We denote by\\
$H^{1,p}(\mathbb{T}; X) = \left\{ u \in L^{p}(\mathbb{T}, X)   : \exists  v \in  L^{p}(\mathbb{T}, X) , \hat{v}(k) = ik \hat{u}(k) for\   \   all\  \  k \in \mathbb{Z}       \right\}$\\
$H^{2,p}(\mathbb{T}; X) = \left\{ u \in L^{p}(\mathbb{T}, X)   : \exists  v \in  L^{p}(\mathbb{T}, X) , \hat{v}(k) = -k^{2} \hat{u}(k) for\   \   all\  \  k \in \mathbb{Z}   \right\}$\\
$H^{3,p}(\mathbb{T}; X) = \left\{ u \in L^{p}(\mathbb{T}, X)   : \exists  v \in  L^{p}(\mathbb{T}, X) , \hat{v}(k) = -ik^{3} \hat{u}(k) for\   \   all\  \  k \in \mathbb{Z}   \right\}$\\
\begin{defi} {\bf \color{green}{\cite{14}}}\\
For $1 \leq p < \infty$, we say that a sequence $\left\{M_{k}\right\}_{k \in \mathbb{Z}} \subset \mathbf{B}(X, Y)$ is an ($L^{p}, H^{3,p}$)-multiplier, if for each $f \in L^{p}(\mathbb{T}, X)$
there exists $u \in H^{3,p}(\mathbb{T}, Y)$ such that $\hat{u}(k) = M_{k}\hat{f}(k)\ \ \text{ for all}\ \ k \in \mathbb{Z}.$
\end{defi}

\begin{lem}\label{l1} {\bf \color{green}{\cite{1}}}\\
Let  $1 \leq p < \infty$ and $(M_{k})_{k \in \mathbb{Z}} \subset  \mathbf{B}(X)\  \ (\mathbf{B}(X)$ is the set of all bounded linear operators from $X$ to $X$). Then the following assertions are equivalent:\\
(i) $(M_{k})_{k \in \mathbb{Z}}$  is an ($L^{p}, H^{1,p}$)-multiplier.\\
(ii) $(ikM_{k})_{k \in \mathbb{Z}}$ is an ($L^{p}, L^{p}$)-multiplier.
\end{lem}
We define 
$$D_{A} : = \frac{d^{3}}{dt^{3}} - A$$
$$\Delta_{k} = (-ik^{3}I + A )\;\; and\;\; \sigma_{\mathbb{Z}}(\Delta) = \left\{ k\in \mathbb{Z} : \Delta_{k} \  is  \   not \  bijective    \right\}$$
We begin by establishing our concept of strong solution for Eq. ${\bf \color{red}{\eqref{e2}}}$

\begin{defi}
Let $f \in L^{p}(\mathbb{T}; X)$. A function $x \in H^{3,p}(\mathbb{T}; X)$ is said to be a $2\pi$-periodic strong $L^{p}$-solution of 
Eq.{\bf \color{red}{\eqref{e2}}} if $x(t) \in  D(A)$ for all $t \geq 0$ and Eq. {\bf \color{red}{\eqref{e2}}} holds almost every where.
\end{defi}

\begin{pr} Let $A$ be a closed linear operator defined on an UMD space $X$. Suppose that \\$\sigma_{\mathbb{Z}}(\Delta) = \phi$ .Then the following assertions are equivalent :
\begin{description}
\item[(i)] $\left(-ik^{3}(-ik^{3}I + A )^{-1}\right)_{k \in \mathbb{Z}}$ is an $L^{p}$-multiplier for $1 < p < \infty$
\item[(ii)] $\left( -ik^{3}(-ik^{3}I + A)^{-1}\right)_{k \in \mathbb{Z}}$ is $R$-bounded.
\end{description}
\end{pr}

\begin{proof}
 (i) $\Rightarrow$ (ii) As a consequence of Proposition {\bf \color{red}{(\ref{p})}}\\
(ii) $\Rightarrow$ (i) Define $M_{k} = -ik^{3} N_{k}$ where $N_{k}=(-ik^{3}I + A)^{-1}$. By Marcinkiewcz Theorem  it  is sufficient to prove that the set $\left\{k(M_{k+1}-M_{k})\right\}_{k \in \mathbb{Z}}$ is $R$-bounded. Since
\begin{align*}
k\left[ M_{k+1} - M_{k}\right]&= k[ - i(k+1)^{3}(-i(k+1)^{3}I + A)^{-1} + ik^{3}(-ik^{3}I + A)^{-1}]\\
&=k [- i(k+1)^{2} N_{k+1} + ik^{2} N_{k}]\\
&=kN_{k+1}[- i(k+1)^{3}((-k^{3}I + A))) + ik^{3}((-i(k+1)^{3}I + A)]N_{k}\\
&=kN_{k+1}[-i(k+1)^{3}A + ik^{3}A)]N_{k}\\
&=ikN_{k+1}[(k^{3}-(k+1)^{3})A]N_{k}\\
&=ikN_{k+1}[-(3k^{2} + 3k + 1)A]N_{k}\\
&=ik^{3}N_{k+1}[-(3 + 3\frac{1}{k} + \frac{1}{k^{2}})](I + ik^{3}N_{k})\\
&=\frac{k^{3}}{(k+1)^{3}}M_{k+1}[-(3 + 3\frac{1}{k} + \frac{1}{k^{2}})](I - M_{k})\\
 \end{align*}
Since products and sums of $R$-bounded sequences is $R$-bounded [{\bf \color{green}{\cite{14}}}. Remark 2.2]. Then the proof is complete.
\end{proof}

\begin{lem}
Let $1 \leq p < \infty$. Suppose that $\sigma_{\mathbb{Z}}(\Delta) = \phi$ and $D_{A}$ is surjective. Then $D_{A}$ is bijective.
\end{lem}

\begin{proof} We have $D_{A}$ is surjective the $H^{3, p}(\mathbb{T}, X)$ to $L^{p}(\mathbb{T}, X)$. Then 
$$\forall f \in L^{p}(\mathbb{T}, X) \ \exists z \in H^{3, p}(\mathbb{T}, X) \ \text{such that } \ D_{A}z = f$$
Suppose that there exists $z_{1}$ and $z_{2}$ such that $D_{A}z_{1} = f$ and $D_{A}z_{2} = f$.  then for $ z  = z_{1} - z_{2} $
we have $D_{A}z = 0$. Taking Fourier transform, we obtain that
$$(-ik^{3} + A)\hat{z}(k) = 0, k \in \mathbb{Z}.$$ i.e $$\Delta_{k}\hat{z}(k)=0$$
It follows that $\hat{z}(k) = 0$ for every $k \in \mathbb{Z}$ and therefore $ z = 0 $. Then 
$ z_{1} = z_{2} $ and $D_{A}$ is bijective.
\end{proof}

\begin{thm}\label{t}
Let $X$ be a Banach space. Suppose that  the operator $D_{A}:= \frac{d^{3}}{dt^{3}} - A$ is an isomorphism of $H^{3,p}(\mathbb{T}, X)$ onto $L^{p}(\mathbb{T}, X)$ for $1 \leq p <  \infty$. Then
\begin{enumerate}
\item for every $k \in \mathbb{Z}$ the operator $\Delta_{k}=(-ik^{3}I + A)$ has bijective,
\item  $\left\{ -ik^{3} \Delta^{-1}_{k}    \right\}_{k \in \mathbb{Z}}$ is $R$-bounded.
\end{enumerate}
\end{thm}
Before to give the proof of Theorem {\bf \color{red}{(\ref{t})}}, we need the following Lemma.

\begin{lem}\label{l2}
if $x \in Ker \Delta_{k}$, then $e^{ikt}x \in Ker D_{A}$
\end{lem}

\begin{proof} $x \in Ker \Delta_{k} \Rightarrow -ik^{3}x = Ax$.\\
Put $z(t) = e^{ikt}x$, then 
\begin{align*}
z'''(t) &= (ike^{ikt}x)''\\
&=(- k^{2}e^{ikt}x)'\\
&=-ik^{3}e^{ikt}x\\
&= e^{ikt}Ax\\
&= Az(t)\\
&\Rightarrow D_{A}z(t) = 0\\
&\Rightarrow e^{ikt}x \in Ker D_{A}.
\end{align*}
\textbf{Proof of Theorem {\bf \color{red}{(\ref{t})}}}:
1) Let $k \in \mathbb{Z}$ and $y \in X$. Then for $f(t) = e^{ikt}y$ , there exists $z \in H^{3, p}(\mathbb{T}; X)$ such that:

$$D_{A} z(t) = f(t)$$
Taking Fourier transform, we deduce that:\\
$(-ik^{3} +  A)\hat{z}(k) = \hat{f}(k)= y \Rightarrow \Delta_{k} = (-ik^{3} +  A)$ is surjective.\\
Let $u \in Ker \Delta_{k}$. By Lemma {\bf \color{red}{\ref{l2}}}, we have $e^{ikt}u \in Ker D_{A}$, then $u = 0$ and  $(-ik^{3} +  A)$ is injective.\\
2) Let $f \in L^{p}(\mathbb{T}; X)$. By hypothesis, there exists a unique $z \in H^{3,p}(\mathbb{T}, X)$ such that $D_{A}z= f$. Taking Fourier transforms, we deduce that $$\hat{z}(k) =  (-ik^{3} +  A )^{-1} \hat{f}(k)\;\;\text{ for all}\;\; k \in \mathbb{Z}.$$ Hence
$$- ik^{3}\hat{z}(k)= -ik^{3}(-ik^{3} +  A)^{-1} \hat{f}(k) \;\;\text{for all}\;\; k \in \mathbb{Z}$$
Since $z \in H^{3,p}(\mathbb{T}; X),$ then there exists $v \in L^{p}(\mathbb{T}; X)$  such that $$\hat{v}(k)=-ik^{3}\hat{z}(k)=  -ik^{3}(-k^{3} +  A)^{-1}\hat{f}(k).$$ Then $\left\{ -ik^{3}\Delta^{-1}_{k}    \right\}_{k \in \mathbb{Z}}$ is an $L^{p}$-multiplier and $\left\{ -ik^{3}\Delta^{-1}_{k}    \right\}_{k \in \mathbb{Z}}$ is $R$-bounded.
\end{proof}

\section{Main result}
Our main result in this section is to establish that the converse of Theorem {\bf \color{red}{\ref{t}}}, are true, provided $X$ is an UMD space.

\begin{lem}\label{l123}
Let  $1 \leq p < \infty$ and $(M_{k})_{k \in \mathbb{Z}} \subset  \mathbf{B}(X)$. Then the following assertions are equivalent:\\
(i) $(M_{k})_{k \in \mathbb{Z}}$  is an ($L^{p}, H^{2,p}$)-multiplier.\\
(ii) $(-k^{2}M_{k})_{k \in \mathbb{Z}}$ is an ($L^{p}, L^{p}$)-multiplier.
\end{lem}

\begin{proof}
We have
\begin{align*}
&(-k^{2}M_{k})_{k \in \mathbb{Z}} \ \text{ is an} \ (L^{p}, L^{p})-multiplier \\
&\Leftrightarrow ik(ikM_{k})_{k \in \mathbb{Z}} \  \text{is an} \ (L^{p}, L^{p})-multiplier \\
&\Leftrightarrow (ikM_{k})_{k \in \mathbb{Z}} \ \text{is an} \  (L^{p}, H^{1,p})-multiplier\  (\text{by Lemma} \ref{l1})
\end{align*}
then the proof ase soon as Lemma (\ref{l1}).
\end{proof}

\begin{lem}\label{l122}
Let  $1 \leq p < \infty$ and $(M_{k})_{k \in \mathbb{Z}} \subset  \mathbf{B}(X)$. Then the following assertions are equivalent:\\
(i) $(M_{k})_{k \in \mathbb{Z}}$  is an ($L^{p}, H^{3,p}$)-multiplier.\\
(ii) $(-ik^{3}M_{k})_{k \in \mathbb{Z}}$ is an ($L^{p}, L^{p}$)-multiplier.
\end{lem}

\begin{thm}\label{um}
Let $X$ be an UMD space and  $A :D(A) \subset X \rightarrow X$ be an closed linear operator. Then the following assertions are equivalent for $1 <  p < \infty$.
\begin{description}
\item[(1)] The operator $D_{A}:= \frac{d^{3}}{dt^{3}} - A$ is an isomorphism of $H^{3,p}(\mathbb{T}, X) \cap L^{p}(D(A), X)$ onto $L^{p}(\mathbb{T}, X)$.
\item[(2)]  $\sigma_{\mathbb{Z}}(\Delta)= \phi$   and      $\left\{ -ik^{3}\Delta^{-1}_{k}    \right\}_{k \in \mathbb{Z}}$ is $R$-bounded.
\end{description}
\end{thm}

\begin{proof}
$1) \Rightarrow2)$ see Theorem {\bf \color{red}{(\ref{t})}}\\
$1) \Leftarrow2)$ Let $f \in L^{p}(\mathbb{T}; X)$ . Define $\Delta_{k} = (-ik^{3}I +  A)$,\\
By Lemma {\bf \color{red}{\ref{l1}}},  the family $\left\{-ik^{3} \Delta^{-1}_{k}\right\}_{k \in \mathbb{Z}}$ is an $L^{p}$-multiplier it is equivalent to\\
the family $\left\{ \Delta^{-1}_{k}\right\}_{k \in \mathbb{Z}}$ is an $L^{p}$-multiplier that maps $L^{p}(\mathbb{T}; X)$ into $H^{3,p}(\mathbb{T}; X)$ (Lemma \ref{l122}),\\
namely there exists $z \in H^{3,p}(\mathbb{T}, X)$ such that
\begin{equation}\label{e4}
\hat{z}(k)=\Delta^{-1}_{k} \hat{f}(k)= (-ik^{3}I +  A )^{-1} \hat{f}(k)
\end{equation}
In particular, $z \in L^{p}(\mathbb{T}; X)$ and there exists $v \in L^{p}(\mathbb{T}; X)$ such that $\hat{v}(k) = -ik^{3} \hat{z}(k)$
By Theorem {\bf \color{red}{\ref{t2}}}, we have
$$z(t) = \lim_{n\rightarrow +\infty} \frac{1}{n+1} \sum^{n}_{m=0} \sum^{m}_{k=-m} e^{ikt}\hat{z}(k)$$
Using now {\bf \color{red}{\eqref{e4}}} we have: 
  $$(-ik^{3}I - A)\hat{z}(k) =  \hat{f}(k)\;\; \mbox{ for all} \;\;k \in \mathbb{Z}.$$
i.e
$$\widehat{(D_{A}z)}(k) = \hat{f}(k)\;\; \mbox{ for all} \;\;k \in \mathbb{Z}.$$
Since $A$ is  closed, then  $z(t) \in D(A)$ and $D_{A}z(t) = f(t)$ [Lemma {\bf \color{red}{\ref{lem2}}}].\\
Uniqueness, suppose that $\exists z_{1}, z_{2} : D_{A}z_{1}(t) = f(t) \ \text{and} \ D_{A}z_{2}(t) = f(t)$. \\
Then \\
$$z = z_{1} - z_{2} \in ker(D_{A}) \  \text{i.e} \ \frac{d^{3}}{dt^{3}}x(t) = Ax(t).$$
Taking Fourier transform, we deduce that:\\
$\Delta_{k}\hat{z}(k)=0 \Rightarrow \hat{z}(k)=0 \ \forall k \in \mathbb{Z} \Rightarrow z = 0$. i.e $z_{1} = z_{2}$. Or $D_{A}$ is linear operator then $D_{A}$ is isomorphism.
\end{proof}

\begin{cor}\label{co}
Let $X$ be an UMD space and  $A :D(A) \subset X \rightarrow X$ be an closed linear operator. Then the following assertions are equivalent for $1 <  p < \infty$.
\begin{description}
\item[(1)] for every $f \in L^{p}(\mathbb{T}; X)$  there exists a unique $2\pi$-periodic strong $L^{p}$-solution of Eq. {\bf \color{red}{\eqref{e2}}}.
\item[(2)]  $\sigma_{\mathbb{Z}}(\Delta)= \phi$   and      $\left\{ -ik^{3}\Delta^{-1}_{k}    \right\}_{k \in \mathbb{Z}}$ is $R$-bounded.
\end{description}
\end{cor}

\begin{proof}
By theorem \ref{um}, we have
\begin{align*}
(2) & \Leftrightarrow D_{A}: H^{3,p}(\mathbb{T}, X) \cap L^{p}(D(A), X) \rightarrow L^{p}(\mathbb{T}, X) \ \text{is an isomorphism}.\\
&\Leftrightarrow \forall f \in L^{p}(\mathbb{T}, X)\ \text{there exits a unique}\ z \in H^{3,p}(\mathbb{T}, X) \cap L^{p}(D(A), X): D_{A}z = f\\
&\Leftrightarrow \forall f \in L^{p}(\mathbb{T}, X)\ \text{there exits a unique}\ z \in H^{3,p}(\mathbb{T}, X) \cap L^{p}(D(A), X): z=D^{-1}_{A}f\\
&\Leftrightarrow \forall f \in L^{p}(\mathbb{T}, X) \text{there exists a unique $2\pi$-periodic strong $L^{p}$-solution of \ } Eq. {\bf \eqref{e2}}.
\end{align*}
\end{proof}

\section{Application}
To apply the provious result, we propose the following partial functional differential equation 
\begin{eqnarray}\label{ef}
\left\{
\begin{array}{ccccc}
\frac{\partial^{3}}{\partial t^{3}}w(t,x)  = \frac{\partial^{2}}{\partial t^{2}}w(t,x) + g(t,x), x \in [0,\pi], t \geq 0,\\\\
w(0,t) = w(\pi,t)=0, t \geq 0.\\\\
w(x,t)=w_{0}(x,t), x \in [0,\pi], -r \leq t \leq 0.
\end{array}
\right.
\end{eqnarray}
Let $X=C_{0}[0,\pi]=\{u \in C([0,\pi],\mathbb{R}) : u(0) = u(\pi) = 0\}$. 
we define the linear operator $A : D(A) \subset X  \rightarrow X$ by 
\begin{center}
$\left\{
\begin{array}{ccccc}
Ay = y''\\\\
D(A) = \{y \in C^{2}([0,\pi],\mathbb{R}): y(0)=y(\pi)=0\}.
\end{array}
\right.$
\end{center}
Put \\
$$y(t)(x) = w(t,x) \ \text{and} \  f(t)(x) = g(t,x)$$
Thus, Eq. (\ref{ef}) takes the following abstract form
\begin{equation}
\frac{d^{3}}{dt^{3}}y(t) = Ay(t)+ f(t)
\end{equation}
It is well known that A is linear operator. Then by [\cite{9}, Section 3.7] (see also references therein), there exists a constant $c > 0$ such that
\begin{center}
$||(- ik^{3} - A)|| \leq \frac{c}{1 + |k^{3}| }$
\end{center}
Then 
$$\sup_{k \in \mathbb{Z}}||k^{3}(- ik^{3} - A)||  < + \infty$$
We deduce from Corollary (\ref{co})  that the above periodic problem has $L^{p}$-strong solution.

\section{Conclusion}
we are obtained necessary and sufficient conditions to guarantee existence and uniqueness of periodic solutions to the equation $z'''(t) = Az(t) +  f(t)$ in terms of either the R-boundedness of the modified resolvent operator determined by the equation. Our results are obtained in the UMD spaces.

\scriptsize\----------------------------------------------------------------------------------------------------------------------------------------\\\copyright
{\it{Copyright International Knowledge Press. All rights reserved.}}

\end{document}